\documentclass[12pt]{article}

\usepackage{amsmath,amsfonts,amssymb,amsthm}
\usepackage{fullpage}
\usepackage[colorlinks]{hyperref} 
\usepackage{enumitem} 
\usepackage{graphicx,tikz}
\usepackage{mathtools}

\DeclareMathOperator{\ex}{ex}
\DeclareMathOperator{\rainbow}{rainbow-}

\newtheorem{theorem}{Theorem}[section]

\theoremstyle{definition}

\theoremstyle{remark}

\newcommand{\rH}{{\rainbow} H}
\newcommand{\rC}{{\rainbow} C}

\hypersetup{
  colorlinks,
  citecolor=blue,
  urlcolor=blue}

\title{Generalized Rainbow Tur\'an  Numbers of Odd Cycles}

\author{
  J\'ozsef Balogh \footnote{Department of Mathematics, University of Illinois Urbana-Champaign, Urbana, Illinois 61801, USA, and Moscow Institute of Physics and Technology, Russian Federation. E-mail: \texttt{jobal@illinois.edu}. Research supported by NSF RTG Grant DMS-1937241, NSF Grant DMS-1764123, Arnold O. Beckman Research Award (UIUC Campus Research Board RB 18132), the Langan Scholar Fund (UIUC), and the Simons Fellowship.}
 \and Michelle Delcourt \footnote {Department of Mathematics,  Ryerson  University,   Toronto,  Ontario  M5B  2K3, Canada. E-mail: \texttt{mdelcourt@ryerson.ca}.  Research supported by NSERC under Discovery Grant No.2019-04269 and an AMS-Simons Travel Grant.}
 \and Emily Heath \footnote {Department of Mathematics, Iowa State University, Ames, IA 50011, USA. Email: \texttt{eheath@iastate.edu}. Research supported by NSF RTG Grant DMS-1937241.}
 \and Lina Li \footnote {Department of Combinatorics and Optimization, University of Waterloo, Waterloo, Ontario N2L 3G1, Canada. Email: \texttt{lina.li@uwaterloo.ca}.}
  }

\begin{document}
\maketitle

\begin{abstract}
Given graphs $F$ and $H$, the \emph{generalized rainbow Tur\'an number} $\ex(n,F,\rH)$ is the maximum number of copies of $F$ in an $n$-vertex graph with a proper edge-coloring that contains no rainbow copy of $H$. B.~Janzer determined the order of magnitude of $\ex(n,C_s,\rC_t)$ for all $s\geq 4$ and $t\geq 3$, and a recent result of O.~Janzer implied that $\ex(n,C_3,\rC_{2k})=O(n^{1+1/k})$. 
We prove the corresponding upper bound for the remaining cases, showing that  $\ex(n,C_3,\rC_{2k+1})=O(n^{1+1/k})$. This matches the known lower bound for $k$ even and is conjectured to be tight for $k$ odd.
\end{abstract}

\section{Introduction}

The \emph{Tur\'an number} of a graph $H$ is the maximum number of edges in an $H$-free graph on $n$ vertices, denoted $\ex(n,H)$. 
This has been generalized in many different ways. For example, the \emph{rainbow Tur\'an number} $\ex^*(n,H)$, introduced in \cite{KMSV}, is the maximum number of edges in a graph on $n$ vertices which can be properly edge-colored with no rainbow copy of $H$. 
Another natural variation is the \emph{generalized Tur\'an number} $\ex(n,F,H)$,  which is the maximum number of copies of a graph $F$ in an $n$-vertex graph that contains no copy of $H$, and was first studied systematically by Alon and Shikhelman~\cite{AS}.
Both of these problems have been extensively studied, see for example~\cite{J2,KMSV} and \cite{AS,B,BG,E,GGMV,GS,GL,Z}.

Gerbner, M\'esz\'aros, Methuku and Palmer~\cite{GMMP}
considered the following generalized problem which unites the two concepts above.
Given two graphs $F$ and $H$, the \emph{generalized rainbow Tur\'an number}, denoted by $\ex(n,F,\rH)$, is the maximum number of copies of $F$ in an $n$-vertex graph which can be properly edge-colored to avoid a rainbow copy of $H$. Note that trivially we have $\ex(n,F,\rH)\geq \ex(n,F,H)$. The question for $F = H$ has been studied for paths, trees, cycles and cliques, see~\cite{GMMP,GJ,J1}.

Recently, B.~Janzer~\cite{J1} determined the order of magnitude of $\ex(n,C_s,\rC_t)$ for all cases except for $s=3$. 
In the case $s=3$, he gave the following bounds.
\begin{theorem}[Janzer~\cite{J1}]\label{thm Janzer}
If $k \geq 2$ is odd then $\ex(n, C_3,\rC_{2k}) = \Omega(n^{1+1/k})$, and if $k$ is even then $\ex(n, C_3,\rC_{2k+1}) = \Omega(n^{1+1/k})$. Furthermore, for every integer $k \geq 2$, we have
\[\ex(n, C_3,\rC_{2k}) = O(\ex^*(n, C_{2k})),\]
and
\[\ex(n, C_3,\rC_{2k}) \geq \ex(n, C_3, C_{2k}) = \Omega(\ex(n, \{C_4, C_6,\dots, C_{2k}\})),\]
\[\ex(n, C_3,\rC_{2k+1}) \geq \ex(n, C_3, C_{2k+1}) = \Omega(\ex(n, \{C_4, C_6,\dots, C_{2k}\})).\]
\end{theorem}  
Very recently, O.~Janzer~\cite{J2} settled a well-known  conjecture of Keevash, Mubayi, Sudakov and Verstra\"{e}te~\cite{KMSV}, proving that 
\begin{equation}\label{eqn even cycle}
\ex^*(n, C_{2k})=\Theta(n^{1 +1/k}).
\end{equation}
Together with Theorem~\ref{thm Janzer}, this yields $\ex(n, C_3,\rC_{2k}) = O(n^{1 +1/k})$, which is tight for $k$ odd.

In this paper, we study the remaining open case for cycles, proving an upper bound on $\ex(n,C_3,\rC_{2k+1})$ which matches the lower bound given by B.~Janzer~\cite{J1} for $k$ even and which is expected to be sharp for $k$ odd as well.

\begin{theorem}\label{thm longer cycles} 
For $k\geq 2$, we have \[\ex(n,C_3,\rC_{2k+1})=O\left(n^{1+1/k}\right).\]
\end{theorem}

In Section~\ref{sec C5}, we give a self-contained proof of the upper bound on $\ex(n,C_3,\rC_5)$ which gives an explicit constant, although it is likely not best possible. In Section~\ref{sec longer cycles}, we prove Theorem~\ref{thm longer cycles} by applying~\eqref{eqn even cycle} to a subgraph in which every rainbow copy of $C_{2k}$ extends to a rainbow copy of $C_{2k+1}$ in the original graph. Throughout the paper, we use $P_k$ to denote the path with $k$ edges.

\section{No rainbow $C_5$}\label{sec C5}

\begin{theorem}\label{thm C5} 
We have $\ex(n,C_3,\rC_5)\leq 32n^{3/2}.$
\end{theorem}

\begin{proof}
Let $G$ be an $n$-vertex graph with a proper edge-coloring $c$ containing no rainbow copy of $C_5$. First, we will show that $G$ contains at most $4|E(G)|$ triangles. 

Fix a vertex $v\in V(G)$ and let $d$ denote the degree of $v$. We will count in two ways the pairs $(S,e)$ where $S\subset N(v)$ contains $\lceil d/2\rceil$  neighbors of $v$ and $e\in E(G[S])$ satisfies $c(e)\neq c(vu)$ for every $u\in S$. There are $\binom{d}{\lceil d/2\rceil}$ ways to choose the set $S$. Throw out any edge in $G[S]$ whose color appears on an edge incident with $v$ and $S$. 
Let $E'$ denote the set of remaining edges in $G[S]$. Now $E'$ must be rainbow $P_3$-free, otherwise we can find a rainbow copy of $C_5$ in $G$. 
Therefore, since Johnston, Palmer, and Sarkar~\cite{JPS} showed that $\ex^*(n,P_3)=\frac{3}{2}n$, 
we have $|E'|\leq \frac{3}{2}\left\lceil \frac{d}{2}\right\rceil$.  Thus, the number of triangles containing $v$ which are formed in this way, that is, the number of pairs of such a set $S$ and an edge from $E'$ of a different color, is at most $\frac{3}{2}\left\lceil \frac{d}{2}\right\rceil\binom{d}{\lceil d/2\rceil}$.

On the other hand, we could instead first choose an edge $e$ in the neighborhood of $v$ to form our triangle, which can be done in $|E(G[N(v)])|$ ways, and then select an additional $\lceil d/2\rceil-2$ vertices from $N(v)$ to form the rest of $S$. However, we do not want the color of $e$ to appear on an edge incident to $v$. Since the edge coloring is proper, this only requires us to throw out at most one vertex from $N(v)$ since at most one edge incident to $v$ can have the same color as $e$. Thus, we can pick the remaining vertices of $S$ in at least $\binom{d-3}{\lceil d/2\rceil-2}$ ways. Therefore, we have 
\[|E(G[N(v)])|\cdot \binom{d-3}{\lceil d/2\rceil-2} \leq \frac{3}{2}\left\lceil \frac{d}{2}\right\rceil\cdot\binom{d}{\lceil d/2\rceil}.\]
Thus, $|E(G[N(v)])|\leq 6d$, so the number of triangles in $G$ is at most \[\frac{1}{3}\sum_{v\in V(G)}|E(G[N(v)])|\leq \frac{1}{3}\sum_{v\in V(G)}6d(v)=4|E(G)|.\]

We may assume that every edge of $G$ is in a triangle, otherwise we could delete an edge without decreasing the number of triangles. Now assume towards a contradiction that $|E(G)|\geq 8n^{3/2}$.

One-by-one, delete vertices of degree less than $4\sqrt{n}$ in $G$.  Note that not all vertices are deleted, since otherwise $|E(G)|<4n^{3/2}$. Denote by $G'$ the remaining induced subgraph with minimum degree $\delta(G')\geq 4\sqrt{n}$, and let $n'=|V(G')|$. Fix an arbitrary vertex $v\in V(G')$.

A \textit{cherry} is a path of length 2. We will form an auxiliary graph $F$ with vertex set $V(F)=N_{G'}(v)$ which contains an edge $uw$ if and only if there are at least seven cherries of the form $uxw$ in $G'$. We will show that $F$ must contain a vertex of degree at least 3 and use this vertex to find a rainbow copy of $C_5$ in $G$.

Let $S\subseteq N_{G'}(v)$ be a set of $\lfloor\sqrt{n}\rfloor$ vertices. We will count the cherries in $G'$ with endpoints in $S$. 
Since each vertex $x$ in $V(G')$ is the center vertex in exactly $\binom{d_S(x)}{2}$ cherries\footnote{$d_S(x)=|N_G(x)\cap S|$.} of this form, we can count the desired cherries as follows:
\begin{align*}
    \sum_{x\in V(G')} \binom{d_S(x)}{2} & \geq n'\binom{\frac{1}{n'}\sum d_S(x)}{2}
     = n'\binom{\frac{1}{n'}\sum_{s\in S} d_{G'}(s)}{2}
     \geq n'\binom{\frac{1}{n'}\cdot\delta(G') |S|}{2}\\
    &  \geq \frac{(\delta(G')|S|)^2}{4n'}
    \geq \frac{16n |S|^2}{4n'}
     > 7\binom{|S|}{2},
\end{align*}
where we use convexity and the fact that $\delta(G')\geq 4\sqrt{n}$.
Thus, there must be some pair of vertices in $S$ which are the endpoints of at least seven cherries in $G'$, and hence, these vertices are adjacent in $F$.  Since $S$ was an arbitrary set of $\lfloor\sqrt{n}\rfloor$ vertices in $N_{G'}(v)=V(F)$, we have shown that  $\alpha(F)\leq \sqrt{n}$, where $\alpha(F)$ denotes the independence number of $F$. This gives $\Delta(F)\geq |F|/\alpha(F)-1\geq 4\sqrt{n}/\sqrt{n}-1\geq 3$, so there is a vertex $u$ in $F$ of degree at least 3. Let $x,y,z$ be neighbors of $u$ in $F$. By assumption, the edge $uv$ is in at least one triangle in $G$, so there is a vertex $w\in V(G)$, possibly in $\{x,y,z\}$, which forms a triangle with $uv$. 

\begin{figure}
\centering
\begin{tikzpicture}[scale=.8]
		\filldraw[very thick, color=black,fill=white] (-6,0) ellipse (0.8 and 2);
		\draw[very thick, black,->] (-4.2,0.4) -- (-3.2,0.4);
		\filldraw[very thick, color=black,fill=white] (0.6,0) ellipse (2.35 and 2);
		\filldraw[very thick, color=black,fill=white] (0.15,0) ellipse (0.7 and 1.8);
		\node[above] at (-6,2.1) {$F$};
		\node[above] at (0.6,2.1) {$G'$};
		\node[above] at (1.65,0.55) {$N_{G'}(v)$};
        \node at (-6.3,-1.4) {$u$};
        \node at (-6.4,-.4) {$z$};
        \node at (-6.4,.4) {$y$};
        \node at (-6.3,1.4) {$x$};
        \node at (-1.5,0) {$v$};
        \node at (-2.2,-2) {$w$};
        \node at (.3,-1.4) {$u$};
        \node at (.4,-.2) {$z$};
        \node at (.4,.5) {$y$};
        \node at (.3,1.4) {$x$};
		\draw[very thick, black, dashed] (-6,-1.2) -- (-6, -.4);
		\draw[very thick, black ,dashed] (-6,-1.2) to [out=15,in=-15] (-6, .4);
		\draw[very thick, black, dashed] (-6,-1.2) to [out=25,in=-25]  (-6, 1.2);
		\draw[very thick, blue] (-1.4,-.4) -- (-1.8,-1.8); 
		\draw[very thick, black] (-1.4,-.4) -- (0, -1.2);  
		\draw[very thick, black] (-1.4,-.4) -- (0, -.4); 
		\draw[very thick, black] (-1.4,-.4) -- (0, .4); 
		\draw[very thick, orange] (-1.4,-.4) -- (0, 1.2); 
		\draw[very thick, red] (-1.8,-1.8) -- (0,-1.2);  
		\draw[very thick, green!70!blue] (0,-1.2) -- (1.9,-.5); 
		\draw[very thick, magenta!60!blue] (0,1.2) -- (1.9,-.5); 
        \node at (-1.9,-1.15) {1}; 
        \node at (-1.1,-1.95) {2}; 
        \node at (1.3,.-1.1) {5}; 
        \node at (1.6,0.2) {4}; 
        \node at (-1,0.6) {3};	
		\draw[fill=black] (-6,-1.2) circle (3pt); 
		\draw[fill=black] (-6,-.4) circle (3pt); 
		\draw[fill=black] (-6,.4) circle (3pt); 
		\draw[fill=black] (-6,1.2) circle (3pt); 
        \draw[fill=black] (-1.4,-.4) circle (3pt); 
        \draw[fill=black] (-1.8,-1.8) circle (3pt); 
        \draw[fill=black] (0,-1.2) circle (3pt); 
        \draw[fill=black] (0,-.4) circle (3pt); 
        \draw[fill=black] (0,.4) circle (3pt); 
        \draw[fill=black] (0,1.2) circle (3pt); 
        \draw[fill=black] (1.9,-.5) circle (3pt); 
    \end{tikzpicture}
\caption{If there is a vertex of degree at least 3 in $F$, then $G$ contains a rainbow copy of $C_5$. Edges in $F$ are dashed while edges in $G$ are solid.}
\label{figureC5}
\end{figure}
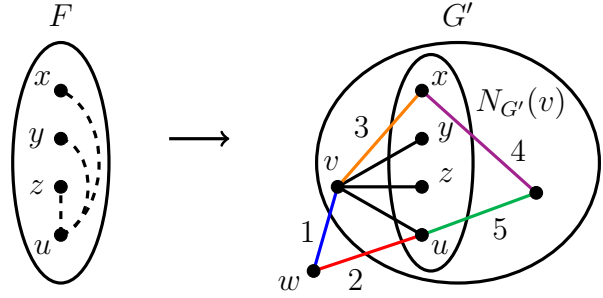

 Let $c(vw)=1$ and $c(uw)=2$. Then since $c$ is a proper coloring, at least one of $vx$, $vy$, and $vz$ is colored with a new color, say $c(vx)=3$. Since $ux$ is an edge in $F$, there are at least seven cherries in $G'$ with endpoints $u$ and $x$, and hence, at least one with new colors 4 and 5 which avoids $v$ and $w$. Thus, there is a rainbow copy of $C_5$ in $G$, and we reach a contradiction. Therefore, $G$ must contain at most $8n^{3/2}$ edges, and hence, at most $32n^{3/2}$ triangles, as desired.
\end{proof}

\section{No rainbow $C_{2k+1}$}\label{sec longer cycles}

\begin{proof}[Proof of Theorem~\ref{thm longer cycles}]
Let $G$ be an $n$-vertex graph with a proper edge-coloring $f:E(G)\rightarrow C$ containing no rainbow copy of $C_{2k+1}$. We may assume each edge in $G$ is in at least one triangle. 

As in Section~\ref{sec C5}, we begin by giving a bound on the number of triangles in $G$ in terms of the number of edges in $G$. Fix a vertex $v\in V(G)$ and let $d$ denote the degree of $v$. Pick a set $S\subset N(v)$ containing $\lceil d/2\rceil$ neighbors of $v$, and throw out any edge in $G[S]$ which is colored using some color which appears on an edge from $v$ to $S$.

Then $G[S]$ cannot contain a rainbow copy of $P_{2k-1}$.
A result of Ergemlidze, Gy\H{o}ri, and Methuku~\cite{EGM} showed that $\ex^*(n, P_{k+1})<\left(\frac{9k}{7} + 2\right)n$.
Therefore, we obtain
\[|E(G[S])|\leq\ex^*\left(\left\lceil\frac{d}{2}\right\rceil,P_{2k-1}\right)\leq \frac{18k-4}{7}\cdot \left\lceil\frac{d}{2}\right\rceil.\]
By counting the triangles containing $v$ and two adjacent vertices in $S$ in two ways, we get \[|E(G[N(v)])|\cdot \binom{d-3}{\lceil d/2\rceil-2} \leq \frac{18k-4}{7}\cdot\left\lceil\frac{d}{2}\right\rceil\cdot\binom{d}{\lceil d/2\rceil}.\]
Thus, $|E(G[N(v)])|\leq \frac{72k-16}{7}d$, and the number of triangles in $G$  is at most \[\frac{1}{3}\sum_{v\in V(G)}|E(G[N(v)])|\leq \frac{72k-16}{21}\sum_{v\in V(G)}d(v)=\frac{144k-32}{21}|E(G)|.\] 
We will show that $|E(G)|=O(n^{1+1/k})$. 

Assume towards a contradiction that $G$ has more edges. For each edge $uv\in E(G)$, arbitrarily fix a vertex $w=w(uv)$ such that $u,v$, and $w$ form a triangle in $G$. We will find a subgraph of $G$ in which any rainbow copy of $C_{2k}$ can be extended to a rainbow copy of $C_{2k+1}$ in $G$.
To this end, randomly select a partition of $V(G)$ into parts $A$ of size $\lfloor\frac{n}{2}\rfloor$ and $B$ of size $\lceil\frac{n}{2}\rceil$. Similarly, take a random partition of $C$ into parts $X$ and $Y$ of sizes $\lfloor\frac{|C|}{2}\rfloor$ and $\lceil\frac{|C|}{2}\rceil$, respectively. 

Let $F$ be the subgraph with vertex set $B$ which contains an edge $uv\in E(G[B])$ if and only if the vertex $w=w(uv)$ is in $A$ and $f(uv)\in X$ while $f(uw), f(vw)\in Y$. 
Then the expected number of edges in $F$ is least $|E(G)|/64$, so we can fix partitions $(A,B)$ and $(X,Y)$ such that the corresponding graph $F$ has at least this many edges.

Note that $F$ inherits the proper edge-coloring $f$ from $G$, so we can apply~\eqref{eqn even cycle} to this subgraph. 
Since $\ex^*(n,C_{2k})=O(n^{1+1/k})$, there must be a rainbow copy of $C_{2k}$ in $F$.
This cycle contains only vertices in $B$, so we can replace an arbitrary edge $uv$ in the cycle by a pair of edges $uw$ and $wv$ with $w\in A$ to create a copy of $C_{2k+1}$ in $G$. Furthermore, the cycle in $F$ is colored only with colors from $X$, while the new edges have colors from $Y$, so we have found a rainbow copy of $C_{2k+1}$.  But this is a contradiction, since $G$ contains no rainbow copies of $C_{2k+1}$. Thus, $|E(G)|=O(n^{1+1/k})$, which implies that the number of triangles in $G$ is $O(n^{1+1/k})$, as desired.
\end{proof}

\begin{section}*{Acknowledgments}
We are grateful to Oliver Janzer for his suggestion which simplified our original proof of Theorem~\ref{thm longer cycles}. We would also like to thank the anonymous referees for their careful reading and useful suggestions.
\end{section}

\bibliography{turan}

\begin{thebibliography}{10}
\expandafter\ifx\csname url\endcsname\relax
  \def\url#1{\texttt{#1}}\fi
\expandafter\ifx\csname urlprefix\endcsname\relax\def\urlprefix{URL }\fi
\expandafter\ifx\csname href\endcsname\relax
  \def\href#1#2{#2} \def\path#1{#1}\fi

\bibitem{KMSV}
P.~Keevash, D.~Mubayi, B.~Sudakov, J.~Verstra{\"e}te, Rainbow {T}ur{\'a}n
  problems, Comb. Probab. Comput. 16 (2007) 109--126.
\newblock \href {https://doi.org/10.1017/S0963548306007760}
  {\path{doi:10.1017/S0963548306007760}}.

\bibitem{AS}
N.~Alon, C.~Shikhelman, Many {$T$} copies in {$H$}-free graphs, J. Comb.
  Theory, Ser. B 121 (2016) 146--172.
\newblock \href {https://doi.org/10.1016/j.jctb.2016.03.004}
  {\path{doi:10.1016/j.jctb.2016.03.004}}.

\bibitem{J2}
O.~Janzer, Rainbow {T}ur\'an number of even cycles, repeated patterns and
  blow-ups of cycles (2020).
\newblock \href {http://arxiv.org/abs/2006.01062} {\path{arXiv:2006.01062}}.

\bibitem{B}
B.~Bollob\'as, On complete subgraphs of different orders, Math. Proc. Camb.
  Philos. Soc. 79~(1) (1976) 19--24.
\newblock \href {https://doi.org/10.1017/S0305004100052063}
  {\path{doi:10.1017/S0305004100052063}}.

\bibitem{BG}
B.~Bollob\'as, E.~Gy\H{o}ri, Pentagons vs. triangles, Discrete Math. 308 (2008)
  4332--4336.
\newblock \href {https://doi.org/10.1016/j.disc.2007.08.016}
  {\path{doi:10.1016/j.disc.2007.08.016}}.

\bibitem{E}
P.~Erd\H{o}s, On the number of complete subgraphs contained in certain graphs,
  Publ. Math. Inst. Hungar. Acad. Sci. 7 (1962) 459--474.

\bibitem{GGMV}
D.~Gerbner, E.~Gy\H{o}ri, A.~Methuku, M.~Vizer, Generalized {T}ur{\'a}n
  problems for even cycles, J. Comb. Theory, Ser. B 145 (2020) 169--213.
\newblock \href {https://doi.org/10.1016/J.JCTB.2020.05.005}
  {\path{doi:10.1016/J.JCTB.2020.05.005}}.

\bibitem{GS}
L.~Gishboliner, A.~Shapira, A generalized {T}ur{\'a}n problem and its
  applications, Int. Math. Res. Not. 2020~(11) (2020) 3417--3452.
\newblock \href {https://doi.org/10.1093/imrn/rny108}
  {\path{doi:10.1093/imrn/rny108}}.

\bibitem{GL}
E.~Gy\H{o}ri, H.~Li, The maximum number of triangles in {$C_{2k+1}$}-free
  graphs, Comb. Probab. Comput. 21 (2012) 187--191.
\newblock \href {https://doi.org/10.1017/S0963548311000629}
  {\path{doi:10.1017/S0963548311000629}}.

\bibitem{Z}
A.~A. Zykov, On some properties of linear complexes, Mat. Sb. (N.S.) 24 (1949)
  163--188.

\bibitem{GMMP}
D.~Gerbner, T.~M\'esz\'aros, A.~Methuku, C.~Palmer, Generalized rainbow
  {T}ur\'an problems (2019).
\newblock \href {http://arxiv.org/abs/1911.06642} {\path{arXiv:1911.06642}}.

\bibitem{GJ}
W.~T. Gowers, B.~Janzer, Generalizations of the {R}uzsa-{S}zemer\'edi and
  rainbow {T}ur\'an problems for cliques, Comb. Probab. Comput. 30~(4) (2021)
  591--608.
\newblock \href {https://doi.org/10.1017/S0963548320000589}
  {\path{doi:10.1017/S0963548320000589}}.

\bibitem{J1}
B.~Janzer, The generalised rainbow {T}ur\'an problem for cycles (2020).
\newblock \href {http://arxiv.org/abs/2005.08073} {\path{arXiv:2005.08073}}.

\bibitem{JPS}
D.~Johnston, C.~Palmer, A.~Sarkar, Rainbow {T}ur{\'a}n problems for paths and
  forests of stars, Electron. J. Comb. 24~(1) (2017) P1.34.
\newblock \href {https://doi.org/10.37236/6430} {\path{doi:10.37236/6430}}.

\bibitem{EGM}
B.~Ergemlidze, E.~Gy\H{o}ri, A.~Methuku, On the rainbow {T}ur\'an number of
  paths, Electron. J. Comb. 26~(1) (2019) P1.17.
\newblock \href {https://doi.org/10.37236/7889} {\path{doi:10.37236/7889}}.

\end{thebibliography}
\end{document}